\documentclass[12pt]{amsart}
\usepackage[letterpaper,body={6.5in,9in}]{geometry}
\usepackage{cite}
\usepackage{amscd}
\usepackage{hyperref}
\usepackage{amsxtra}
\usepackage{amssymb}
\usepackage{enumerate}
\usepackage{amsmath} 
\usepackage{amsthm}
\usepackage{url}
\usepackage[all]{xy}

\usepackage{graphicx}
\usepackage{epstopdf}
\usepackage[usenames]{color}

 \theoremstyle{plain}
 \newtheorem{thm}{Theorem}
 
 \newtheorem{lem}[thm]{Lemma}

 \theoremstyle{definition}
 \newtheorem{defn}[thm]{Definition}

 \theoremstyle{remark}
 \newtheorem{rem}[thm]{Remark}

\newcounter{nootje}
\newcommand\noot[1]
  {\marginpar{\tiny\begin{minipage}{20mm}\begin{flushleft}\color{red}\thenootje : #1\end{flushleft}\end{minipage}}\addtocounter{nootje}{1}}
 \newcommand{\abs}[1]{\left|#1\right|}
 \newcommand{\set}[1]{\left\{#1\right\}}

 \newcommand{\p}[1]{\left(#1\right)}

 \newcommand{\N}{\mathbb{N}}
 \newcommand{\Z}{\mathbb{Z}}

 \newcommand{\ZZ}{\mathbb{Z}}

\begin{document}
\title{Primitive prime divisors in zero orbits of polynomials}
\author{Kevin Doerksen}
\address{Department of Mathematics,
         Simon Fraser University,
         Burnaby, BC,
         Canada V5A 1S6}
\email{kdoerkse@gmail.com}
\author{Anna Haensch}
\address{Department of Mathematics and Computer Science, 
		Wesleyan University, 
		Middletown CT, 06459, 
		USA}
\email{ahaensch@wesleyan.edu}

\date{September 20, 2010}

\begin{abstract}
Let $(b_n) = (b_1, b_2, \dots)$ be a sequence of integers.  A primitive prime divisor of a term $b_k$ is a prime which divides $b_k$ but does not divide any of the previous terms of the sequence.  A zero orbit of a polynomial $\varphi(z)$ is a sequence of integers $(c_n)$ where the $n$-{th} term is the $n$-th iterate of $\varphi$ at 0.  We consider primitive prime divisors of zero orbits of polynomials.  In this note, we show that for $c, d$ in $\ZZ$, where $d \geq 2$ and $c \neq \pm 1$, every iterate in the zero orbit of $\varphi(z) = z^d + c$ contains a primitive prime whenever zero has an infinite orbit.  If $c = \pm 1$, then every iterate after the first contains a primitive prime.
\end{abstract}
\maketitle

\section{Introduction}
A \emph{dynamical system} is a pair $(\varphi,S)$ where $S$ is a set and $\varphi$ is a map from $S$ to itself.  Given such a pair, the \emph{orbit} of an element $\alpha\in S$ under $\varphi$ is the set
\[
\{\varphi(\alpha),\varphi^2(\alpha),...,\varphi^n(\alpha),...\}
\]
where
\[
\varphi^n(x)=\underbrace{\varphi\circ\varphi\circ...\circ\varphi}_{n\text{ times}}(x).
\]
Such an element $\alpha$ can classified according to the size of the orbit.  If the orbit contains only finitely many values, then $\alpha$ is a \emph{preperiodic point}.  If the orbit contains infinitely many values, then $\alpha$ is a \emph{wandering point}.  If we restrict to the case where $S = \ZZ$ and $\varphi\in\ZZ[z]$, the orbit of a wandering point $\alpha$, will yield an infinite sequence of integers.  Some very natural questions about prime factorization and divisibility in these sequences arise.  In particular, one can ask which iterates in the orbit contain prime divisors not dividing any previous term.
\begin{defn}
Let $(b_n) = (b_1, b_2, \dots)$ be a sequence of integers.  We say that the term $b_n$ contains a \emph{primitive prime divisor} if there exists a prime $p$ such that $p \mid b_n$, but $p \nmid b_i$ for $i < n$.
\end{defn}   

Questions about terms containing primitive prime divisors have been asked for a number of different recurrence sequences.  Classical results by Bang \cite{AB} (for $b=1$) and Zsigmondy \cite{KZ} showed that for any $a, b \in \N$, every term in the sequence $a^n - b^n$ has a primitive prime divisor past the sixth term.
The question of primitive prime divisors in second-order linear recurrence sequences was completely solved by Bilu, Hanrot, and Voutier in \cite{BHV}.  

Recent papers have addressed the question of primitive prime divisors in nonlinear recurrence sequences.  Elliptic divisibility sequences, for example, were considered by Silverman in \cite{JS88}, and later by Everest, Mclaren, and Ward in \cite{EMW} and Yabuta in \cite{Y}.  In our paper, we consider recurrence sequences generated by the orbit of wandering points of non-linear polynomials.  This question was first addressed by Rice \cite{Rice}.  

\begin{thm}(Rice 2007) 
Let $\varphi(z)\in\ZZ[z]$ be a monic polynomial of degree $d\geq 2$.  Suppose that $(b_n) = \varphi^n(0)$ has infinite orbit under iteration of $\varphi$ such that $(b_n)$ is a rigid divisibility sequence.  Then all but finitely many terms of the sequence $(\alpha, \varphi(\alpha), \varphi^2(\alpha), \dots)$ contain a primitive prime divisor.  
\end{thm}

See Section~\ref{S:RDS} for a definition of rigid divisibility sequences. 
Rice also showed that if zero is a preperiodic point of a monic polynomial of degree $\geq 2$, then the orbit of any wandering point has finitely many terms which contain no primitive prime divisor.  Silverman and Ingram \cite{PI} later generalized this result to arbitrary rational maps over number fields.  Faber and Granville \cite{FaberGranville2011} also considered rational maps, $\phi$, over number fields, but they looked at primitive prime divisors in the sequence generated by $\p{\phi^{n + \Delta}(\alpha) - \phi^n(\alpha)}$ for a wandering point $\alpha$ and a fixed integer $\Delta \geq 1$.

Silverman and Ingram use Roth's theorem to prove their result, and therefore their proof does not give a means to find an effective upper bound on the terms without primitive prime divisors.  Rice also remarks that though his bounds are effectively computable, he does not compute them.  Silverman \cite{JS} proposed that it would be of interest to compute explicit upper bounds on $n$ for the terms which contain no primitive prime divisor, when the polynomial $\varphi(z)$ and $\alpha$ are fixed.  In this paper, we answer this question for a certain class of polynomials. We prove

\begin{thm}\label{T:Main}
Let $\varphi \in \ZZ[z]$ be the polynomial $\varphi(z) = z^d + c$, where $c, d \in \ZZ$ and $d \geq 2$.  Suppose that zero is a wandering point of $\varphi$ and write $b_n = \varphi^n(0)$.  Then    
\begin{enumerate}
	\item If $c = \pm 1$, then $b_n$ contains a primitive prime for all $n \geq 2$.
	\item For all other $c \in \Z$, $b_n$ contains a primitive prime for all $n \geq 1$.
\end{enumerate}
\end{thm}

We prove Theorem~\ref{T:Main} in two parts.  We begin by showing that for the sequence $(b_n)$, defined in the statement of the theorem, there is an upper bound on the size of the product of all prime divisors of a term $b_n$ which are not primitive prime divisors.  We then show that the sequence grows too fast for any one term to not contain a primitive prime divisor (other than possibly the first term).

\medskip
\noindent\emph{Acknowledgements.}
The authors would like to thank Joe Silverman for originally drawing their attention to the problem.  They would also like to thank Michelle Manes and Rafe Jones for the helpful comments and suggestions.  Thanks also goes to the organizers of the NSF-funded Arizona Winter School, at which the initial research for this paper took place.

\section{Rigid Divisibility Sequences}\label{S:RDS}

In order to prove Theorem~\ref{T:Main}, we make use of a special type of divisibility sequence, with terminology taken from Jones \cite{RJ} and Rice \cite{Rice}.  For $\alpha \in \ZZ$, let $v_p(\alpha)$ denote the valuation at $p$ of $\alpha$.  A sequence $(b_n)$ of integers is said to be a \emph{rigid divisiblity sequence} if for every prime $p$ the following two properties hold:
\begin{enumerate}
\item\label{L:RDS1} If $v_p(b_n) > 0$ then $v_p(b_{nk}) = v_p(b_n)$ for all $k\geq 1$, and
\item\label{L:RDS2} If $v_p(b_n) > 0$ and $v_p(b_m) > 0$ then $v_p(b_n) = v_p(b_m) = v_p(b_{\gcd(m,n)})$.
\end{enumerate} 

\begin{lem}[Rice {\cite{Rice}}]\label{T:RDS}
Let $\varphi \in \ZZ[z]$ be the polynomial $\varphi(z) = z^d + c$, where $c, d \in \ZZ$ and $d \geq 2$.  Let zero be a wandering point of $\varphi$ and write $b_n = \varphi^{n}(0)$.  Then $(b_n)$ is a rigid divisibility sequence.
\end{lem}

\begin{proof}
Let $p$ be a prime and suppose $v_p(b_n) = e > 0$ for some $n, e \in \N$.  Then $b_n = p^em$ for some $m$ where $p \nmid m$.  Then 
$$
b_{n+1} = p^{ed} m^d + c = p^{2e}\p{p^{d-2}m^d} + c \equiv c \pmod{p^{e+1}},
$$
with the last congruence true because $2e \geq e+1$.  But $b_1 = c$ so 
$$
b_{n+1} \equiv b_1 \pmod{p^{e+1}}.
$$
By induction on $t$, we have $b_{n+t} \equiv b_t \pmod{p^{e+1}}$, and so in general for $k \geq 1$,
$$
b_{kn + r} \equiv b_r \pmod{p^{e+1}},
$$
and in particular, for $r = 0$, we get $v_p(b_{kn}) = v_p(b_{n})$.  

Now suppose $m, n \in \N$ such that $v_p(b_m) > 0$ and $v_p(b_n) > 0$.  Without loss of generality, suppose $m < n$ and $m \!\nmid\! n$ (the case where $m \vert n$ has already been covered).  Let $s, t \in \N$ such that $t \geq 1$ and $sm + tn = \gcd(m,n)$.  Then
$$
b_{\gcd(m,n)} = b_{sm + tn} \equiv b_{tn} \equiv b_n \pmod{p^{e+1}},
$$
therefore $v_p\!\p{b_{\gcd(m,n)}} = v_p\!\p{b_n}$, and since $m$ is a positive multiple of $\gcd(m,n)$, we also conclude $v_p\!\p{b_{\gcd(m,n)}} = v_p\!\p{b_m}$.
\end{proof}

\begin{rem}
Rice actually proves a more general result to Lemma~\ref{T:RDS}.  In Propositions~3.1 and 3.2 from \cite{Rice}, he shows that for any polynomial $\varphi$ of degree $d \geq 2$ that has a wandering orbit at zero, then the sequence $b_n$ as defined in Lemma~\ref{T:RDS} is a rigid divisibility sequence if and only if the coefficient of the linear term of $\varphi$ is zero.
\end{rem}

Suppose $(b_n)$ is a rigid divisibility sequence.  For every $n$, we can write 
$$
b_n = p_1^{e_{1}}p_2^{e_{2}} \dots p_k^{e_{k}} q_{1}^{f_{1}} \dots q_{\ell}^{f_{\ell}}
$$
where $p_i$ are the primitive primes of $b_n$ and $q_j$ are the primes of $b_n$ which are not primitive.  Let 
\begin{align*}
P_n &= p_1^{e_{1}} \dots p_{k}^{e_{k}} = \text{the \emph{primitive part} of }b_n \text{ and}\\
N_n &= q_{1}^{f_{1}} \dots q_{\ell}^{f_{\ell}} = \text{the \emph{non-primitive part} of }b_n.
\end{align*}

\begin{lem}\label{T:NP}
Let $(b_n)$ be a rigid divisibility sequence and let $P_n$ and $N_n$ be as above.  Then 
$$
N_n = \prod_{d\mid n, d \neq n} P_d.
$$
\end{lem}

\begin{proof}
Let $p$ be a prime divisor in the non-primitive part of $b_n$.  Then there exists some positive integer $d < n$ such that $p$ is a primitive prime divisor of $b_d$.  By Property~\ref{L:RDS2} of rigid divisibility sequences,  $v_{p}(b_n) = v_{p}(b_d) = v_{p}(b_{\gcd(d, n)})$.  Since $p$ is a primitive prime divisor of $b_d$, we must have $\gcd(d, n) \geq d$ and so $d \mid n$.  Therefore $N_n$ divides $\prod_{d\mid n, d\neq n} P_d$.

Now suppose $d \mid n$ and suppose $q$ is a primitive prime divisor of $b_d$.  Then by Property~\ref{L:RDS1}, $v_q(b_d) = v_q(b_n)$.  Therefore the product $\prod_{d\mid n, d\neq n} P_d$ divides $N_n$, completing the proof.  
\end{proof}

Armed with these two results, we are now able to proceed with the main theorem of this paper.

\section{Proof of Main Result}
We begin this section with a useful lemma.
\begin{lem}\label{Lem:Increasing}
Let $\varphi \in \ZZ[z]$ be the polynomial $\varphi(z) = z^d + c$, where $c, d \in \ZZ$ and $d \geq 2$.  Let $a \in \ZZ$, and for nonnegative $n \in \ZZ$, define the sequence $(b_{a, n})$ by
$$
b_{a, 0} = a \quad \text{and} \quad b_{a, n+1} = \varphi(b_{a, n})
$$
and let $B_{a,n} = \abs{b_{a,n}}$.  If $\abs{a} \geq \abs{c}$ and $\abs{a} > 2$ then $(B_{a, n})$ is an increasing sequence.
\end{lem}

\begin{proof}
We prove this by induction. For the base case,
$$
B_{a, 1} = \abs{\varphi(a)} = \abs{a^d + c} > \abs{2a} - \abs{c} \geq \abs{a}.
$$
Now suppose $(B_{a, n})$ is increasing on $n \leq N$.  Then
$$
B_{a, N+1} = \abs{\varphi(b_{a, N})} = \abs{(b_{a, N})^d + c} > \abs{2b_{a, N}} - \abs{c} > B_{a, N} + \abs{a} - \abs{c} \geq B_{a, N},
$$
completing the proof.
\end{proof}

The statement of Theorem~\ref{T:Main} requires zero to be a wandering point of $\varphi$.  In the next lemma, we characterize all polynomials over $\ZZ$ of the form $z^d + c$ for which zero is a preperiodic point.  Rice proves a more general result by giving a complete classification of monic polynomials for which the orbit of zero is finite (see \cite[Proposition~2.1]{Rice}).  Nevertheless, the special case where $\varphi(z) = z^d + c$ is a relevant lemma with a straightforward proof.  We therefore provide a full proof.

\begin{lem}\label{Lem:Wandering}
Let $\varphi \in \ZZ[z]$ be the polynomial $\varphi(z) = z^d + c$, where $c, d \in \ZZ$ and $d \geq 2$. Then either\begin{enumerate}
	\item Zero is a wandering point and the sequence $(B_n)$, defined by $B_n = \abs{\varphi^n(0)}$, is an increasing sequence, or
	\item Zero is a preperiodic point and exactly one of the following is true
	\begin{enumerate}
		\item $c = 0$,
		\item $c = -1$ and $d$ is even, or
		\item $c = -2$ and $d = 2$.
	\end{enumerate}
 \end{enumerate}
\end{lem}

\begin{proof}
Note that $\varphi(0) = c$, so if $c \notin \set{0, \pm 1, \pm 2}$, then $(B_n)$ is an increasing sequence by Lemma~\ref{Lem:Increasing}, and so zero must be a wandering point.

For $c > 0$, a simple induction shows that $(\varphi^n(0))$ is an increasing sequence, and so zero is a wandering point.

If $c = 0$ then $\varphi(0) = 0$ and zero is a preperiodic point.  

Now suppose $c = -1$. If $d$ is even, then $\varphi(0) = -1$ and $\varphi(-1) = 0$ and therefore zero is a preperiodic point.  If $d$ is odd then $\varphi(0) = -1$, $\varphi(-1) = -2$, and $\varphi(-2) = (-2)^d - 1$.  Since $\abs{\varphi(-2)} > 2$, we can apply Lemma~\ref{Lem:Increasing} to show that all subsequent iterates grow in absolute value.

Finally, suppose $c = -2$.  If $d = 2$ then $\varphi(0) = -2$, $\varphi(-2) = 2$, and $\varphi(2) = 2$.  If $d >2$ then $\varphi(0) = -2$ and $\varphi(-2) = (-2)^d - 2$.  But 
$$
\abs{-2^d - 2} \geq 2^d - 2 \geq 2^3 - 2 = 6.
$$
We can therefore apply Lemma~\ref{Lem:Increasing} to conclude that zero is a wandering point.
\end{proof}

We now are ready to prove Theorem~\ref{T:Main}.

\begin{proof}[Proof of Theorem~\ref{T:Main}]
Note first that if $c = 0$, then zero would not be a wandering point, so we must have $c \ne 0$.  Also, $b_1 = \varphi(0) =  c$, so $b_1$ will have a primitive prime if and only if $c \neq \pm 1$.  For $b_2$, note that 
\begin{align*}
b_2 &= \varphi(b_1) = c^d + c\\
	&= c (c^{d-1} + 1),
\end{align*}
and since $b_1 = c$, we see that $b_2$ will contain a primitive prime divisor, except when $c = 0$ or when $c = -1$ and $d$ is even.  In both cases,  by Lemma~\ref{Lem:Wandering}, zero would not be a wandering point.

Now let $m \in \N$ with $m \geq 3$.   We will prove that $b_m$ contains a primitive prime.  Let  $\abs{\cdot}$ denote the Euclidean absolute value.  Then  
\begin{align*}
\abs{b_m} &= \abs{\p{b_{m-1}}^d + c}\\ 
	&\geq \abs{\p{b_{m-1}}^d} - \abs{c}\\
	&\geq \abs{b_{m-1}}^2 - \abs{b_1}	&& \text{because }b_1 = \varphi(0) = c \text{ and } d\geq 2\\
	&> \abs{b_{m-1}}^2 - \abs{b_{m-1}}. && \text{because } (b_n) \text{ is increasing and } m \geq 3.
\end{align*}
We can factor the last line to obtain
\begin{equation}\label{E:Factored}
\abs{b_m} > \abs{b_{m-1}} \cdot \p{\abs{b_{m-1}} - 1}.
\end{equation}

To complete the proof, we first need to show that for all $m \geq 3$, 
$$
\prod_{k=1}^{m-1} \abs{b_k} < \abs{b_m}.
$$
We prove this claim by induction.  The base case is trivially true.  Now assume that $\prod_{k=1}^{m-2} \abs{b_k} < \abs{b_{m-1}}$.  In particular, this implies 
\begin{equation}\label{E:Product}
\prod_{k=1}^{m-2} \abs{b_k} \leq \abs{b_{m-1}} - 1.
\end{equation}
Combining  \eqref{E:Product} with \eqref{E:Factored}, 
\begin{align*}
\prod_{k=1}^{m - 1} \abs{b_k}
	&= \abs{b_{m - 1}}\cdot \prod_{k=1}^{m-2} \abs{b_{k}}\\
	&\leq \abs{b_{m - 1}}  \cdot \p{\abs{b_{m-1}} - 1}\\
	&< \abs{b_{m}}.
\end{align*}

Finally, by Lemma~\ref{T:RDS},  we know that $(b_n)$ is a rigid divisibility sequence.  For all $m \in \N$, let $P_m$ and $N_m$ denote the primitive part and the non-primitive part of $b_m$ respectively.  Then $\abs{b_m} = P_{m}N_{m}$ and by Lemma~\ref{T:NP}
$$
N_m = \prod_{d\vert m, d \neq m} P_d
	\leq \prod_{k=1}^{m-1} P_k
	\leq \prod_{k=1}^{m-1} \abs{b_k}
	<\abs{b_m}
$$
Therefore $P_m > 1$ and $b_m$ contains a primitive prime.
\end{proof}

\end{document}